\documentclass{amsart}

\usepackage{amsmath,amsthm,amssymb,amscd}
\usepackage[all]{xy}
\usepackage[dvips]{graphicx}
\usepackage{enumerate}
\usepackage{mathrsfs}

\theoremstyle{plain}
\numberwithin{equation}{section}
\newtheorem{thm}{Theorem}[section]

\newtheorem{prop}[thm]{Proposition}

\newtheorem{lem}[thm]{Lemma}
\theoremstyle{definition}
\newtheorem{dfn}[thm]{Definition}
\newtheorem{exm}[thm]{Example}

\newtheorem{rem}[thm]{Remark}

\def\rank{\mathop{\mathrm{rank}}\nolimits}
\def\dim{\mathop{\mathrm{dim}}\nolimits}

\def\End{\mathop{\mathrm{End}}\nolimits}

\def\Ad{\mathop{\mathrm{Ad}}\nolimits}
\def\AdG{\mathop{\mathrm{Ad}^G}\nolimits}
\def\ad{\mathop{\mathrm{ad}}\nolimits}

\def\CC{\mathbb{C}}

\def\bfd{\mathbf{d}}

\def\RR{\mathbb{R}}
\def\ZZ{\mathbb{Z}}
\def\PP{\mathbb{P}}

\def\ii{\mathbf{i}}
\def\EE{\mathtt{E}}
\def\ttE{\mathtt{E}}

\def\fil{F^{\bullet}}
\def\ttH{\mathtt{H}}

\def\ttS{\mathtt{S}}
\def\bfc{\mathbf{c}}

\def\bd{\mathrm{bd}}
\def\calE{\mathcal{E}}

\def\calZ{\mathcal{Z}}

\def\calK{\mathcal{K}}
\def\wf{W_{\bullet}}
\def\cl{\mathrm{c\ell}}

\def\frakb{\mathfrak{b}}
\def\frakg{\mathfrak{g}}
\def\frakk{\mathfrak{k}}
\def\frakt{\mathfrak{t}}
\def\frakp{\mathfrak{p}}

\def\frakh{\mathfrak{h}}

\def\calO{\mathcal{O}}
\def\diff{\mathrm{d}}
\def\Re{\mathop{\mathrm{Re}}\nolimits}
\def\Im{\mathop{\mathrm{Im}}\nolimits}
\def\Int{\mathop{\mathrm{int}}\nolimits}

\def\Aut{\mathop{\mathrm{Aut}}\nolimits}

\def\ev{\mathrm{ev}}

\def\od{\mathrm{od}}

\def\sr{\mathcal{S}}
\def\rt-m{\sqrt{m}}
\def\00{\mathbf{0}}
\def\Ker{\mathop{\mathrm{Ker}}\nolimits}

\title[Cycle connectivity and pseudoconcavity of flag domains]{Cycle connectivity and pseudoconcavity\\ of flag domains}
\author[T.~Hayama]{Tatsuki Hayama}
\address{Department of Business Administration, Senshu University, Higashimita 2-1-1, Tama, Kawasaki, Kanagawa 214-8580, Japan}
\email{hayama@isc.senshu-u.ac.jp}

\begin{document}
\maketitle
\begin{abstract}
We prove that a non-classical flag domain is pseudoconcave if it satisfies a certain condition on the root system.
Moreover, we prove that every point in a codimension-one real boundary orbit of a non-classical period domains is a pseudoconcave boundary point if it satisfies a certain Hodge-theoretical condition.
\end{abstract}
\section{Introduction}

Flag domains are open real group orbits in flag manifolds.
The simplest example of flag domain is open orbits in $\PP^1(\CC)$.
The projective line $\PP^1(\CC)$ is a flag manifold with the holomorphic action of $SL(2,\CC)$, which has the three real forms $SL(2,\RR)$, $SU(1,1)$, and $SU(2)$.
Here the upper/lower-half planes are open $SL(2,\RR)$-orbits, the unit disks around $0$ and $\infty$ are open $SU(1,1)$-orbits, and $\PP^1(\CC)$ itself is the $SU(2)$-orbit.
In general, there are finitely many open real group orbits in flag manifolds, and study on flag domains has been developed by complex geometers (cf. \cite{FHW}).
Flag domains are classified into two kinds; {\it classical} or {\it non-classical}.
Here we say a flag domain is non-classical if it has no non-constant holomorphic function.
A classical flag domain is almost like a Hermitian symmetric domain and is well-studied, however a non-classical one is not.
In this paper, we investigate non-classical flag domains.
In particular, we focus on their {\it cycle connectivity} and {\it pseudoconcavity}.

Cycle connectivity of flag domains are investigated by Huckleberry \cite{Huc} and Green, Robles and Toledo \cite{GRT}.
We recall it briefly.
Let $\check{D}$ be a flag manifold with a holomorphic action of a connected complex semisimple Lie group $G$.
Let $D$ be an open $G_{\RR}$-orbit contained in $\check{D}$ with a real form $G_{\RR}$.
For a base point $o\in D$, we have the parabolic subgroup $P$ of $G$ stabilizing $o$ and $G_{\RR}\cap P$ contains a $\theta$-stable fundamental Cartan subgroup with a Cartan involution $\theta$.
For the maximal compact subgroup $K_{\RR}$ fixed by $\theta$, the $K_{\RR}$-orbit $C_0=K_{\RR}o$ is a compact submanifold contained in $D$.
A flag domain $D$ is non-classical if and only if any two points of $D$ are connected by a connected chain $g_1C_0\cup \cdots\cup g_{\ell}C_0$  contained in $D$ with $g_1,\ldots ,g_{\ell}\in G$.

The other property we discuss in this paper is pseudoconcavity.
We say $D$ is pseudoconcave if $D$ contains a relatively compact subset where a Levi form on every boundary point has at least one negative eigenvalue. 
Pseudoconcavity of complex manifolds is studied by Andreotti and Grauert in 1960's.
Pseudoconcave manifolds are a generalization of compact complex manifolds, which behave like compact complex manifolds in many ways.
For instance, if $D$ is pseudoconcave, $H^0(D,\calO)=\CC$ and $\dim{H^0(D,E)}<\infty$ (the  finiteness theorem) for any holomorphic vector bundle $E$.
Moreover, the field of all meromorphic functions on a pseudoconcave manifold $D$  is an algebraic field of transcendental degree less than $\dim{D}$.
Huckleberry \cite{Huc0} proved that classical flag domains are not pseudoconcave but pseudoconvex.

Cycle connectivity and pseudoconcavity are closely related. 
Huckleberry introduced {\it one-connectivity} in \cite{Huc0}.
One-connectivity is cycle connectivity in a strong sense, which requires that any two points are connected by a single cycle $gC_0$ with some $g\in G$.
He showed that $D$ is pseudoconcave if $D$ is generically one-connected.
For example, $D$ is generically one-connected if $G_{\RR}=SL(n,\RR)$, however we do not completely know what kind of flag domain this property holds.
In \cite[Lecture 10]{G}, Green, Griffiths and Kerr examined pseudoconcavity of the Carayol domain (Example \ref{SU21}) by an argument about boundary behavior of its cycle space. 
On the other hand, Koll\'{a}r \cite{Ko} proved the finiteness theorem for every non-classical flag domains by using their cycle connectivity without pseudoconcavity.
Huckleberry conjectured that any non-classical flag domain is pseudoconcave. 

Our purpose of this paper is to give a sufficient condition for non-classical flag domains to be pseudoconcave, which \cite{Huc0} did not reach.  
In the Lie algebra level, we have the $\theta$-stable Cartan subalgebra $\frakh_{\RR}\subset \frakg_{\RR}\cap \frakp$ and the maximal compact subalgebra $\frakk_{\RR}$ fixed by $\theta$. 
Let $\Delta$ be the set of roots.
Then we have the graded Lie algebra decomposition $\frakg=\frakg^{-k}\oplus\cdots \oplus \frakg^{k}$ defined by a subset of the simple roots such that $\frakp=\bigoplus_{\ell\geq 0}\frakg^{\ell}$.
We say a root $\alpha\in \Delta$ is compact (resp. non-compact) if $\alpha\in\Delta(\frakk)$ (resp. $\alpha\in\Delta(\frakk^{\perp})$).
\begin{thm}\label{main}
Suppose that there exist a compact root $\beta$ such that for any non-compact root $\alpha$ in $\Delta(\frakg^{<0})$ the $\beta$-string containing $\alpha$ is one of the following:
\begin{itemize}
\item $\{\alpha, \alpha+\beta\}$ with $\alpha+\beta\in \Delta(\frakp)$;
\item $\{\alpha, \alpha+\beta, \alpha+2\beta\}$ with $\alpha+2\beta\in \Delta(\frakp)$.
\end{itemize}
Then $D$ is pseudoconcave.
\end{thm}
 \medskip

Our study is motivated by Hodge theory.
A rational Hodge structure defines a rational algebraic group  called a Mumford--Tate group, and an orbit of a real Mumford--Tate group is called a Mumford--Tate domain, which is a measurable flag domain (cf. \cite{GGK}).
For example, a period domain (cf. \cite{CG}) is a Mumford--Tate domain of a special kind.    
We give two examples of Mumford--Tate domain where the above theorem holds (Example \ref{SU21}--\ref{SO21}).

In the latter half of this paper, we discuss pseudoconcavity at a boundary point of $D$ in a real codimension-one $G_{\RR}$-orbit for a non-classical period domain $D$.
In this case $G_{\RR}$ is $Sp(2n)$ or $SO(p,q)$ depending on its Hodge numbers.
Now the boundary $\bd{(D)}$ in $\check{D}$ is the disjoint union of finitely many boundary $G_{\RR}$-orbits.
For a real codimension-one $G_{\RR}$-orbit $\calO$, we may consider a Levi form and define pseudoconcavity at a point in $\calO$.
By transitivity of $G_{\RR}$ acting on $D$, every point in $\calO$ is pseudoconcave boundary point of $D$ if one point in $\calO$ is.
We then say a codimension-one $G_{\RR}$-orbit $\calO$ is a pseudoconcave boundary orbit of $D$ if a point in $\calO$ is pseudoconcave boundary point of $D$.

Boundary $G_{\RR}$-orbits have a relationship with degeneration of Hodge structures.
By the nilpotent orbit theorem \cite{S}, degenerating Hodge structures over a product of punctured disks are asymptotically approximated by nilpotent orbits. 
A nilpotent orbit is generated by a pair consisting of $N\in \frakg_{\RR}$ and $\fil\in \check{D}$, and the limit point $\lim_{\Im{(z)}\to\infty}\exp{(zN)}=:\fil_{\infty}\in\check{D}$, which is called the reduced limit, is in a boundary $G_{\RR}$-orbit. 
A boundary $G_{\RR}$-orbit is said to be poralizable if it is a $G_{\RR}$-orbit of a reduced limit.
Remark that {\it not} every boundary $G_{\RR}$-orbits are polarizable.

Green, Griffiths and Kerr \cite{G} and Kerr and Pearlstein \cite{KP} showed that any codimension-one boundary $G_{\RR}$-orbit $\calO$ is polarizable.
We then have a nilpotent orbit $(N,\fil)$ so that the reduced limit $\fil_{\infty}$ is in $\calO$, which is called a {\it minimal degeneration}.
Minimal degenerations are studied by Green, Griffiths and Robles \cite{GGR}. 
In Theorem \ref{prop-per}, we will give a sufficient condition of $(N,\fil)$ for $\calO$ to be a pseudoconcave boundary orbit examining minimal degenerations.

\medskip  

In this paper, we do not deal with degree of pseudoconcavity, and we leave it open.
If a flag domain $D$ is $q$-pseudoconcave in the sense of Andreotti-Grauert, we have the finiteness theorem for cohomology groups of degree less than $\dim{D}-q-1$ with locally free sheaves.
In addition, it is speculated that the finiteness theorem is true for degree less than $\dim{C_0}$ in \cite[Lecture 10]{G} as an open question.

\section{Pseudoconcavity of complex manifolds}\label{pcc}
We recall pseudoconcavity of complex manifolds following \cite{A}.
Let $\Omega$ be an open subset of $\CC^n$ with a smooth boundary.
For every point $z_0$ in the boundary $\bd{ (\Omega)}$, we can find a neighborhood $U_{z_0}$ of $z_0$ and a $C^{\infty}$-function $\phi:U_{z_0}\to \RR$ such that
$$(\diff \phi)_{z_0}\neq 0,\quad \Omega\cap U_{z_0}=\{z\in U_{z_0}\; |\; \phi(z)<\phi(z_0)\}.$$ 
The real tangent plane at $z_0$ to $\bd{(\Omega)}$ contains the $(n-1)$-dimensional complex plane defined by the equation
$$\sum_k \frac{\partial \phi}{\partial z^k}(z_0)(z^k-z_{0}^k)=0$$
with the coordinate function $z^1,\ldots ,z^n$.
This is called the analytic tangent plane.
Since $\phi$ is real-valued, the quadratic form 
$$L(\phi)_{z_0}(z):=\sum_{k,\ell}\frac{\partial^2 \phi}{\partial z^k\partial \bar{z}^{\ell}}(z_0)z^k\bar{z}^\ell,$$
is Hermitian, which is called the Levi form.
The boundary point $z_0$ is said to be pseudoconcave if the Levi form has at least one negative eigenvalue on the analytic tangent plane. 
Remark that the number of positive/negative eigenvalues does not depend on the choice of coordinate function and defining function.
We may assume  $z_0=0$ and $\phi(0)=0$, and we may change the coordinate so that the Tayler expansion of $\phi$ at $0$ is
$$\phi(z)=2\Re(z^1)+L_0(\phi)(z)+O(\|z\|^3)$$
and the Levi form restricted to the analytic tangent plane $T_0(\bd{(\Omega)})$, defined by $z^1=0$, is 
$$L(\phi)_{0}|_{T_0(\bd{(\Omega)})}(z)=\sum_{k\geq 2}^n \lambda_k |z^k|^2$$
with eigenvalues $\lambda_2,\ldots ,\lambda_n$.
If $0$ is a pseudoconcave boundary point, we have at least one negative eigenvalue, then we have a holomorphic map $\rho :\mathbb{D}\to \cl{(\Omega)}$ from the unit disk such that   
$\rho(0)=0$ and $\bd{(\rho(\mathbb{D}))}\subset \Omega$.
On the other hand, if $0$ is not a pseudoconcave boundary point, we do not have such a holomorphic map.
We then define pseudoconcavity of complex manifolds as follows:
\begin{dfn}\label{pcc-dfn}
Let $X$ be a connected complex manifold.
We say $X$ is pseudoconcave if there exists a relatively compact open subset $U$ with a smooth boundary such that at every point $z\in \bd{(U)}$ a holomorphic map $\rho:\mathbb{D}\to \cl{(U)}$ where $\rho(0)=z$ and $\bd{(\rho(\mathbb{D}))}\subset U$ exists.   
\end{dfn}
\begin{exm}
\begin{enumerate}
\item Every compact complex connected manifold is pseudoconcave.
\item Let $Z$ be a compact complex connected manifold of dimension greater than $2$ and let $Y$ be a complex submanifold with $\dim{Y}\leq \dim{Z}-2$. Then $Z- Y$ is pseudoconcave.
\end{enumerate}
\end{exm}
\section{Flag domains}
We review flag domains and their root structure, and then we prove Theorem \ref{main}.
This proof is based on technique of \cite[\S 3.2]{Huc0}.
The key point of Huckleberry's proof is to construct a relatively compact neighborhood of the base cycle $C_0$ which is filled out by cycles and where any points are one-connected to $C_0$.
We construct such a neighborhood in Lemma \ref{thm}. 
Here Cayley transform associated with compact root plays important role.
\subsection{Parabolic subalgebras}
Let $\frakg$ be a complex semisimple Lie algebra, and let $\frakp$ be a parabolic subgroup.
Let $\frakg_{\RR}$ be a real form of $\frakg$.
We fix a Cartan subalgebra $\frakh_{\RR}\subset \frakg_{\RR}\cap \frakp$ and choose a Borel subgroup $\frakb\subset \frakp$ containing the Cartan subgroup $\frakh:=\frakh_{\RR}\otimes \CC$.
The Cartan subalgebra $\frakh$ determines a set $\Delta=\Delta(\frakg ,\frakh)\subset \frakh^*$ of roots.
A root $\alpha\in\Delta$ defines the root space $\frakg^{\alpha}$, and we have the root space decomposition $\frakg=\frakh\oplus\bigoplus_{\alpha\in\Delta}\frakg^{\alpha}$.
For a subalgebra $\mathfrak{s}$, we denote by $\Delta(\mathfrak{s})$ the subset of roots of which root spaces contained in $\mathfrak{s}$.
The Borel subalgebra $\frakb$ defines a positive root system by  
$$\Delta^+:=\Delta(\frakb)=\{\alpha\in\Delta\; |\; \frakg^{\alpha}\subset \frakb\}.$$  

Let $\sr=\{\sigma_1,\ldots ,\sigma_r\}$ be the set of simple roots, and let $\{\ttS^1,\ldots , \ttS^r\}$ be the dual basis to $\sr$.
An integral linear combination $\EE=\sum_{j}n_j\ttS^j$ is called a grading element.
The $\EE$-eigenspaces
$$\frakg^{\ell}=\bigoplus_{\alpha(\EE)=\ell}\frakg^{\alpha}, \quad \frakg^0=\frakh\oplus \bigoplus_{\alpha(\EE)=0}\frakg^{\alpha}$$
determines a graded Lie algebra decomposition $\frakg=\frakg^{-k}\oplus \cdots \oplus \frakg^{k}$ in the sense that $[\frakg^{\ell},\frakg^m]\subset \frakg^{\ell +m}$.
A grading $\EE$ defines a parabolic subalgebra $\frakp_{\EE}=\bigoplus_{\ell\geq 0}\frakg^{\ell}$.
On the other hand, setting
$$I(\frakp)=\{i\; |\; \frakg^{-\sigma_i}\not\subset \frakp\},$$
the parabolic subalgebra $\frakp_{\EE}$ defined by $\EE=\sum_{i\in I(\frakp)}\ttS^i$ coincides with $\frakp$.

We review Chevalley basis and their properties (cf. \cite[\S 25]{Hum}).
Since the Killing form defines a non-degenerate negative-definite symmetric form on $\frakh_{\RR}$, we have the induced form $(\bullet , \bullet)$ on $\frakh_{\RR}^*$.
Let $\alpha\in\Delta\cup\{0\}$ and let $\beta\in\Delta$.
The set of all members of $\Delta\cup\{ 0\}$ of the form $\alpha+n\beta$ for $n\in\ZZ$ is called the $\beta$-string containing $\alpha$.
Then the $\beta$-string containing $\alpha$ is given by
\begin{align}\label{string}
\{\alpha+n\beta\;|\; -r\leq n\leq q\}.
\end{align}
If $\alpha$ and $\beta$ are linearly independent, we have 
$$\langle\alpha,\beta\rangle:=2\frac{( \alpha, \beta)}{(\beta ,\beta)}=r-q.$$
We can choose $\ttH^{\alpha}\in\frakh$ and $x^{\alpha}\in\frakg^{\alpha}$ for all $\alpha\in \Delta$ satisfying
\begin{align*}
&[x^{\alpha},x^{-\alpha}]=\ttH^{\alpha},\quad [\ttH^{\beta},x^{\alpha}]=\langle \alpha,\beta\rangle x^{\alpha},\\
&[x^{\alpha},x^{\beta}]=
	\begin{cases}
	c_{\alpha,\beta}x^{\alpha +\beta}&\text{if }\alpha+\beta\in\Delta\\
	0&\text{if }\alpha +\beta\not\in \Delta
	\end{cases}\\
&\text{where }c_{\alpha,\beta}=-c_{-\alpha,-\beta}\in\ZZ.
\end{align*}
Here $c_{\beta,\alpha}=\pm(r+1)$ if $\alpha$ and $\beta$ are linearly independent and $\alpha+\beta\in\Delta$.
Now
$$\{x^{\alpha}\; |\; \alpha\in\Delta\}\cup \{\ttH^{\sigma}\; |\; \sigma\in\sr\}$$ 
is a basis of $\frakg$, which is called a Chevalley basis.
\begin{lem}[{\cite[\S 25.2]{Hum}}]\label{Jac}
If $\alpha$ and $\beta$ are linearly independent, then
$$[x^{-\beta},[x^{\beta},x^{\alpha}]]=q(r+1)x^{\alpha}$$
where $(r,q)$ is given by the form (\ref{string}).
\end{lem}

\begin{exm}[$\frakg =\mathfrak{sl}_2(\CC)$]\label{SL2-1}
A basis of $\frakg_{\RR} =\mathfrak{sl}_2(\RR)$ is given by
\begin{align*}
x^{\alpha}  =\begin{pmatrix}
0&1\\0&0
\end{pmatrix},\quad
\ttH^{\alpha} =\begin{pmatrix}
1&0\\ 0&-1
\end{pmatrix},\quad
x^{-\alpha} =
\begin{pmatrix}
0&0\\ 1&0
\end{pmatrix}.
\end{align*}
Here $\frakh=\CC \ttH^{\alpha}$ is a Cartan subalgebra, and $x^{\pm\alpha}$ is in the root space $\frakg^{\pm\alpha}$ where $\alpha$ is the root given by $\alpha(\ttH^{\alpha})= 2$.
This triple satisfies
\begin{align}\label{SL2-triple}
[x^{\alpha},x^{-\alpha}]=\ttH^{\alpha}, \quad [\ttH^{\alpha},x^{\alpha}]= 2 x^{\alpha}, \quad [\ttH^{\alpha},x^{-\alpha}]= -2 x^{-\alpha}.
\end{align}
In this case, $\Delta=\{\pm\alpha\}$ and we set $\alpha$ as a positive root.
We define a grading $\ttE=\ttS$ where $\ttS$ is the dual of $\alpha$.
Then 
$$\frakg^1=\frakg^{\alpha},\quad \frakg^0=\frakh, \quad \frakg^{-1}=\frakg^{-\alpha},$$
and $\frakp_{\ttE}=\frakg^{1}\oplus \frakg^{0}$.
\end{exm}
In general, a triple in a semisimple Lie algebra satisfying (\ref{SL2-triple}) is called a standard $sl_2$-triple. 

\subsection{Flag domains and Cayley transforms}
In the Lie group level, we have the parabolic subgroup $P\subset G$ corresponding to $\frakp\subset \frakg$.
The homogeneous manifold $\check{D}=G/P$ is called a flag manifold. 
We fix a Cartan involution of $\frakg_{\RR}$.
Let $\frakp_x$ be the Lie algebra of the parabolic subgroup $P_x$ of $G$ stabilizing $x\in \check{D}$.
By \cite[Theorem 4.2.2]{FHW}, the $G_{\RR}$-orbit $G_{\RR}x$ is open if and only if $\frakp_x=\frakp_{\ttE}$ where
\begin{itemize}
\item $\frakg_{\RR}\cap\frakp_x$ contains a fundamental Cartan subalgebra $\frakh_{\RR}$, and
\item $\ttE$ is an integral linear combination of a set of simple roots for a system $\Delta^+\subset \Delta=\Delta(\frakg, \frakh)$ with $\frakh=\frakh_{\RR}\otimes \CC$ such that $\tau\Delta^+=-\Delta^+$ for the complex conjugation $\tau$.
\end{itemize}

An open $G_{\RR}$-orbit is called a flag domain.
Let $D$ be a flag domain contained in $\check{D}$, and fix a base point $o\in D$.
We may assume $\frakp=\frakp_o$.
There exists a fundamental Cartan subalgebra $\frakh_{\RR}\subset \frakg_{\RR}\cap \frakp$, a positive root system $\Delta^+$, and a grading $\ttE$ such that $\frakp=\frakp_{\ttE}$.
We may choose a Chevalley basis of $\frakg$.
We define the Cayley transform $\bfc_{\alpha}$ for $\alpha\in \Delta$ by
$$\bfc_{\alpha}=\exp{(\frac{\pi}{4}(x^{-\alpha}-x^{\alpha}))}.$$

\begin{rem}
Usually, Cayley transforms of the above form are defined for noncompact imaginary roots as \cite[\S VI.7]{Kn}.
We use a Cayley transform for a {\it compact} root in the proof of Theorem \ref{main}.
\end{rem}
\begin{exm} [$\frakg =\mathfrak{sl}_2(\CC)$]\label{SL2-2}
This is continuation of Example \ref{SL2-1}.
Let us consider the three real forms $\mathfrak{su}(2)$, $\mathfrak{su}(1,1)$, and $\mathfrak{sl}(2,\RR)$ of $\mathfrak{sl}(2,\CC)$.
We set
\begin{align*}
&u^{\alpha}:=x^{\alpha}-x^{-\alpha}=\begin{pmatrix}0&1\\-1&0\end{pmatrix}, \quad h^{1}:=\ii\ttH^{\alpha}=\ii\begin{pmatrix}1&0\\0&-1\end{pmatrix},\\
&v^{\alpha}:=\ii(x^{\alpha}+x^{-\alpha})=\ii\begin{pmatrix}0&1\\1&0\end{pmatrix},
\end{align*}
where $\ii=\sqrt{-1}$.
Then 
\begin{align*}
&\mathfrak{su}(2)=\RR u^{\alpha}+\RR h^1+\RR v^{\alpha},
\quad\mathfrak{su}(1,1)=\RR \ii u^{\alpha}+\RR h^1+\RR \ii v^{\alpha},\\
&\mathfrak{sl}(2,\RR)=\RR x^{\alpha}+\RR\ttH^{\alpha} +\RR x^{-\alpha}=\RR (u_{\alpha}-\ii v_{\alpha})+\RR \ii h^1+ \RR (u_{\alpha}+\ii v_{\alpha}).
\end{align*}
The $G$-orbit $\check{D}$ is $\PP^1(\CC)$ in this case.
Each real group orbits of $o=\begin{pmatrix}1\\0\end{pmatrix}\in\check{D}$ are 
\begin{align*}
\PP^1(\CC),\quad
\left\{\begin{pmatrix}1\\z\end{pmatrix}\; ;\;|z|<1\right\},\quad
\left\{\begin{pmatrix}1\\y\end{pmatrix}\; ;\;y\in\RR\right\}\cup \left\{\begin{pmatrix}0\\1\end{pmatrix}\right\}.
\end{align*}
In a case where $\frakg_{\RR}$ is $\mathfrak{su}(2)$ or $\mathfrak{su}(1,1)$, the Cartan subalgebra $\frakh_{\RR}=\frakp\cap\frakg_{\RR}=\RR h^1$ is compact and $\tau x^{\alpha}=x^{-\alpha}$ for the complex conjugate $\tau$.
On the other hand, in the case where $\frakg_{\RR}$ is $\mathfrak{sl}(2,\RR)$, the Cartan subalgebra $\frakh_{\RR}=\frakp\cap\frakg_{\RR}=\RR\ii h^1$ is noncompact and $\tau x^{\alpha}=x^{\alpha}$.
The Cayley transform for $\alpha$ is
$$\bfc_{\alpha}=\frac{1}{\sqrt{2}}\begin{pmatrix}1&-1\\1&1\end{pmatrix}.$$
We then have
$$\bfc_{\alpha}(o)=\begin{pmatrix}1\\1\end{pmatrix},\quad \bfc^2_{\alpha}(o)=\begin{pmatrix}0\\1\end{pmatrix}.$$
 \end{exm}

\begin{prop}\label{prop}
If $\alpha$ and $\beta$ are linearly independent, then 
\begin{align*}
&\Ad{(\bfc_{-\beta}^2)}x^{\alpha}=
\begin{cases}
\pm x^{\alpha+\beta}&\text{if }(r,q)=(0,1),\\
\pm x^{\alpha+2\beta}&\text{if }(r,q)=(0,2)
\end{cases}
\end{align*}
where $(r,q)$ is given by the form (\ref{string}).
\end{prop}
\begin{proof}
First, we consider the case for $(r,q)=(0,1)$.
Since $\alpha-\beta\not\in\Delta$,
$$\ad{(x^{\beta}-x^{-\beta})}x^{\alpha}=\ad{(x^{\beta})}x^{\alpha}=c_{\beta,\alpha}x^{\alpha+\beta}.$$
By Lemma \ref{Jac}, 
$$[x^{-\beta},[x^{\beta},x^{\alpha}]]=x^{\alpha}.$$
Then 
$$(\ad{(x^{\beta}-x^{-\beta})})^2x^{\alpha}=-[x^{-\beta},[x^{\beta},x^{\alpha}]]=-x^{\alpha}.$$
Therefore
\begin{align*}
\Ad{(\exp{(\frac{\pi}{2}(x^{\beta}-x^{-\beta}))})}x^{\alpha}&=\sum_{\ell=0}^{\infty}\frac{(-1)^{2\ell}}{(2\ell)!}(\frac{\pi}{2})^{2\ell}x^{\alpha}+\sum_{k=0}^{\infty}\frac{(-1)^k}{(2k+1)!}(\frac{\pi}{2})^{2k+1} c_{\beta,\alpha}x^{\alpha+\beta}\\
&=\cos{(\frac{\pi}{2})}x^{\alpha}+\sin{(\frac{\pi}{2})} c_{\beta,\alpha}x^{\alpha+\beta}=c_{\beta,\alpha}x^{\alpha+\beta}
\end{align*}
Since $c_{\beta,\alpha}=\pm 1$, the equation holds.

Next, we consider the case for $(r,q)=(0,2)$.
As is the case for $(r,q)=(0,1)$, we have
\begin{align*}
\ad{(x^{\beta}-x^{-\beta})}x^{\alpha}=c_{\beta,\alpha}x^{\alpha+\beta},\quad [x^{-\beta},[x^{\beta},x^{\alpha}]]=2x^{\beta}.
\end{align*}
Then
$$(\ad{(x^{\beta}-x^{-\beta})})^2 x^{\alpha}=c_{\beta,\alpha+\beta}c_{\beta,\alpha}x^{\alpha+2\beta}-2x^{\alpha}.$$
By Lemma \ref{Jac},
\begin{align*}
[x^{\beta},[x^{-\beta},x^{\alpha +2\beta}]]=2x^{\alpha +2\beta}.
\end{align*}
On the other hand,
$$[x^{\beta},[x^{-\beta},x^{\alpha +2\beta}]]=c_{\beta ,\alpha +\beta}c_{-\beta ,\alpha +2\beta}x^{\alpha +2\beta}.$$
Then $c_{\beta ,\alpha +\beta}c_{-\beta ,\alpha +2\beta}=2$.
Therefore,
$$(\ad{(x^{\beta}-x^{-\beta})})^3 x^{\alpha}=-c_{-\beta,\alpha+2\beta}c_{\beta,\alpha+\beta}c_{\beta,\alpha}x^{\alpha+\beta}-2c_{\beta,\alpha}x^{\alpha+\beta}=-4c_{\beta,\alpha}x^{\alpha+\beta}.$$
To summarize the above calculations, we have
\begin{align*}
\Ad{(\exp{(\frac{\pi}{2}(x^{\beta}-x^{-\beta}))})}x^{\alpha}=x^{\alpha}&+\sum_{k=0}^{\infty}\frac{1}{(2k+1)!}(\frac{\pi}{2})^{2k+1}(-4)^k c_{\beta,\alpha}x^{\alpha+\beta}\\
&+\sum_{\ell=1}^{\infty}\frac{1}{(2\ell)!}(\frac{\pi}{2})^{2\ell}(-4)^{\ell -1}(c_{\beta,\alpha+\beta}c_{\beta,\alpha}x^{\alpha+2\beta}-2x^{\alpha}),
\end{align*}
where
\begin{align*}
&\sum_{k=0}^{\infty}\frac{1}{(2k+1)!}(\frac{\pi}{2})^{2k+1}(-4)^k =\frac{1}{2}\sum_{k=0}^{\infty}\frac{(-1)^k}{(2k+1)!}(\frac{\pi}{2})^{2k+1}2^{2k+1} =\frac{1}{2}\sin{\pi}=0,\\
&\sum_{\ell=1}^{\infty}\frac{1}{(2\ell)!}(\frac{\pi}{2})^{2\ell}(-4)^{\ell -1}=-\frac{1}{4}\sum_{\ell=1}^{\infty}\frac{(-1)^{\ell}}{(2\ell)!}(\frac{\pi}{2})^{2\ell}2^{2\ell }=-\frac{1}{4}(\cos{\pi}-1)=\frac{1}{2}.
\end{align*}
Hence
$$\Ad{(\exp{(\frac{\pi}{2}(x^{\beta}-x^{-\beta}))})}x^{\alpha}=\frac{1}{2}c_{\beta,\alpha+\beta}c_{\beta,\alpha}x^{\alpha+2\beta}.$$
Since $c_{\beta,\alpha+\beta}=\pm 2$ and $c_{\beta,\alpha}=\pm 1$, the equation holds.
\end{proof}
In the case for $(r,q)=(1,0)$ or $(r,q)=(2,0)$, the same equation holds if $\beta$ is replaced by $-\beta$.

\begin{rem}
The proof of Proposition \ref{prop} looks similar to the one of \cite[Proposition 6.72]{Kn}.
However, the latter one requires that $\beta$ is orthogonal to $\alpha$, i.e. $r=q$.
\end{rem}

\subsection{Proof of Theorem \ref{main}}\label{proof}
There exists a Cartan involution $\theta$ of $\frakg_{\RR}$ such that $\frakh_{\RR}$ is a $\theta$-stable fundamental Cartan subalgebra.
Let $\frakg_{\RR}=\frakk_{\RR}\oplus \frakk_{\RR}^{\perp}$ be the Cartan decomposition associated with $\theta$.
We set the base cycle $C_0=K_{\RR}o$ with the maximal compact subgroup $K_{\RR}$ corresponding to $\frakk_{\RR}$.
By \cite[Theorem 4.3.1]{FHW}, the base cycle $C_0$ is a compact complex manifold satisfying $C_0=Ko$ with $K=K_{\RR}\otimes \CC$.
\begin{lem}\label{thm}
Suppose that there exists a compact root $\beta$ satisfying the condition of Theorem \ref{main}.
We choose a sufficiently small  $\varepsilon >0$ so that 
$$
\calK:=\left\{\begin{array}{l|l}
k\prod_{i=1}^{\ell}\exp{(\varepsilon_i x^{\alpha_i})}z
&
\begin{array}{l}
k\in K_{\RR},\; z\in C_0\\
\sum_i|\varepsilon_i|^2\leq \varepsilon
\end{array}
\end{array}\right\}
$$
 is contained in $D$ where $\{\alpha_1,\ldots \alpha_{\ell}\}=\Delta(\frakk^{\perp})\cap\Delta (\frakg^{<0})$.
Then $\Int{(\calK)}$ is a relatively compact subset containing $C_0$.
Moreover, for every point $z'$ in $\calK$, there exists $gC_0\subset \calK$ with $g\in G$ such that $C_0\cap gC_0 \neq\emptyset$ and $z'\in gC_0$.
\end{lem}
\begin{proof}
Let $U_o$ be an open neighbourhood of $o$ in $C_0$.
Since $\{\prod_{i=1}^{\ell}\exp{(\varepsilon_i x^{\alpha_i})}z\; ;\; |\varepsilon_i|<\varepsilon,\; z\in U_o\}$ is an open neighborhood of $o$ in $D$ and $K_{\RR}$ acts on $C_0$ transitively, $\Int{(\calK)}$ is a relatively compact open subset containing $C_0$.
Let $z'=k\xi z\in\calK$ where $k\in K_{\RR}$, $z\in C_0$, and
$$\xi=\prod_{i=1}^{\ell}\exp{(\varepsilon_i x^{\alpha_i})}$$
with $\sum_i|\varepsilon_i |^2\leq \varepsilon$.
Since $z'\in k\xi C_0$, it is enough to show $k\xi C_0\cap C_0\neq \emptyset$.
Now we have
\begin{align*}
\AdG{(\bfc_{-\beta}^2)}\xi =\prod_{i=1}^{\ell}\exp{(\varepsilon_i \Ad{(\bfc^2_{-\beta})}x^{\alpha_i})}
\end{align*}
where $\AdG$ is the adjoint action on $G$.
By Proposition \ref{prop} and the hypothesis, $\Ad{(\bfc^2_{-\beta})}x^{\alpha_i}\in\frakp$, therefore $\AdG{(\bfc_{-\beta}^2)}\xi\in P$.
Then $\AdG{(\bfc_{-\beta}^2)}(\xi) o=o$, and hence $\xi \bfc_{\beta}^2(o)=\bfc_{\beta}^2(o)$.
Since $\bfc_{\beta}\in K$ and $K$ acts on $C_0$ transitively, $\bfc_{\beta}^2(o)\in C_0\cap \xi C_0$.
It concludes that $k\bfc_{\beta}^2(o)\in C_0\cap k\xi C_0$.
\end{proof}

By applying the above $\mathcal{K}$ to the proof of \cite[Theorem 3.7]{Huc0}, we complete the proof.
To show pseudoconcavity at any point $z\in\mathrm{bd}(\mathcal{K})$, we need to construct a disk $\mathbb{D}$ about $z$ such that $\bd{(\mathbb{D})}\subset \mathrm{int}(\mathcal{K})$.
By the above lemma, we have $ gC_0\subset\mathcal{K}$ with $g\in G$ containing $z$ such that there exists a point $z_0\in gC_0\cap C_0$.
We may construct $Z\subset gC_0$ with $z\in Z$ and $Z\cong \PP^1(\CC)$ so that $Z$ intersects an arbitrary neighborhood of $z_0$.
We then choose $y_0\in \mathrm{int}(\mathcal{K})\cap Z$ and define $\mathbb{D}$ to be the complement in $Z$ of the closure of a sufficiently small disk about $y_0$ so that $\mathrm{bd}(\mathbb{D})\subset \mathrm{int}(\mathcal{K})$.
Therefore $z$ is a pseudoconcave boundary point, and hence $D$ is pseudoconcave.

\begin{exm}\label{SU21}
Let $D$ be the Mumford--Tate domain with $G_{\RR}=SU(2,1)$ investigated by Carayol \cite{C}.
The root diagram is depicted in Figure \ref{root1}, where compact roots are those within a box and the shaded area is a Weyl Chamber.
\begin{figure}[h]
\begin{center}
\includegraphics{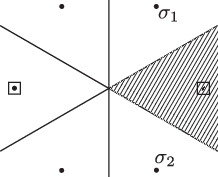}
\caption{The root diagram of $\mathfrak{sl}(3,\CC)$}
\label{root1}
\end{center}
\end{figure}
The parabolic subalgebra $\frakp$ is associated with $\ttE=\ttS^1+\ttS^2$.
Here $\{-\sigma_1,-\sigma_2\}=\Delta(\frakk^{\perp})\cap\Delta(\frakg^{<0})$, and $\sigma_1+\sigma_2\in\Delta(\frakk)$ satisfies the condition of Theorem \ref{main}.
Hence $D$ is pseudoconcave.
\end{exm}
\begin{exm}\label{SO21}
Let $D$ be the period domain with $h^{2,0}=2$ and $h^{1,1}=1$.
Then $G_{\RR}=SO(4,1)$, and the root diagram is  depicted in Figure \ref{root2}.
\begin{figure}[h]
\begin{center}
\includegraphics{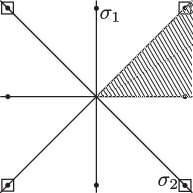}
\caption{The root diagram of $\mathfrak{so}(5,\CC)$}
\label{root2}
\end{center}
\end{figure}
The parabolic subalgebra $\frakp$ is associated with $\ttE=\ttS^1$.
Here $\{-\sigma_1,-\sigma_1-\sigma_2\}=\Delta(\frakk^{\perp})\cap\Delta(\frakg^{<0})$, and $2\sigma_1+\sigma_2\in\Delta(\frakk)$ satisfies the condition.
Hence $D$ is pseudoconcave.
\end{exm}

\section{Period domains}
We review minimal degeneration for period domains and prove Theorem \ref{prop-per}.
Strategy of this proof is similar to the one of Theorem \ref{main}.
We will show that a certain boundary point and an interior point  are connected by a single cycle $gC_0$, which induces pseudoconcavity at the boundary orbit.
We examine $sl_2$-triples associated with minimal degenerations for this proof.

In this section, $D$ is a period domain parametrizing Hodge structures of weight $n$ with Hodge numbers $\{h^{p,q}\}_{p+q=n}$ on a real vector space $V_{\RR}$ polarized by $Q$.
Here 
$$
G_\RR=\Aut{(V_{\RR},Q)}\cong
\begin{cases}
Sp(m,\RR)& \text{if }n\text{ is odd,}\\
SO(m_{\ev},m_{\od})& \text{if }n\text{ is even}
\end{cases}
$$
where $2m=\dim{V}$ if $n$ is odd, $m_{\ev}=\sum_{p\text{: even}}h^{p,q}$ and  $m_{\od}=\sum_{p\text{: odd}}h^{p,q}$ if $n$ is even.
We then have
$$
D\cong
\begin{cases}
Sp(m,\RR)/\prod_{p<k}U(h^{p,q})& \text{if }n=2k +1,\\
SO(m_{\ev},m_{\od})/(\prod_{p<k}U(h^{p,q})\times SO(h^{k,k}))& \text{if }n=2k.
\end{cases}
$$
If we choose a  reference point in $D$, then the $G_{\RR}$-orbit $D$ is a flag domain contained in the $G$-orbit with $G=\Aut{(V_{\CC},Q)}$.  
\subsection{Nilpotent orbits and $SL_2$-orbits}\label{nilp-sl2}
We recall the nilpotent orbit theorem and the $SL_2$-orbit theorem of \cite{S}.
A pair $(N,F^{\bullet})$ consisting of a nilpotent $N\in \frakg_{\RR}$, as an element of $\End{(V_{\RR})}$, and $F^{\bullet}\in\check{D}$ is called a nilpotent orbit if it satisfies the following conditions:
\begin{itemize}
\item $\exp{(zN)}F^{\bullet}\in D$ if $\Im{(z)}$ is sufficiently large;
\item $NF^{p}\subset F^{p-1}.$
\end{itemize}
For a nilpotent orbit $(N,\fil)$, there exists the monodromy weight filtration $\wf:=W_{\bullet}(N)[-n]$, and $(\wf ,\fil)$ is a mixed Hodge structure, which is called the limit mixed Hodge structure. 

Let $(N,\fil)$ be a nilpotent orbit such that  $(\wf,\fil)$ is split over $\RR$, i.e. the Deligne decomposition 
$$V_{\CC}=\bigoplus_{p,q}I^{p,q}$$
is defined over $\RR$.
We define $Y\in\frakg_{\RR}$ which acts on $I^{p,q}$ by the scalar $p+q-n$ and $N^+\in\frakg_{\RR}$ so that $(N^+,Y,N)$ is a $sl_2$-triple in $\frakg_{\RR}$.
The $SL_2$-orbit theorem guarantees existence of a homomorphism $v:SL(2,\RR)\to G_{\RR}$ such that
\begin{align*}
v_*\begin{pmatrix}
0&1\\ 0&0
\end{pmatrix}
=N^+,\quad
v_*\begin{pmatrix}
1&0\\ 0&-1
\end{pmatrix}
=Y,\quad
v_*\begin{pmatrix}
0&0\\1&0
\end{pmatrix}
=N
\end{align*}
and a $SL_2(\CC)$-equivalent horizontal holomorphic map $\psi:\PP^1(\CC)\to\check{D}$ given by
$$\begin{pmatrix}1\\z\end{pmatrix}\mapsto \exp{(zN)}\fil $$
Here $\exp{(zN)}\fil\in D$ for $\Im{(z)}>0$, and $\psi$ defines a $SL_2(\RR)$-equivalent map from the upper half plane to $D$.

Let
$$\bfd_N :=\exp{(\ii\frac{\pi}{4}(N^+ +N))}.$$
We then have
$$\varphi:=\bfd_N(\fil)=\psi(\frac{1}{\sqrt{2}}\begin{pmatrix}1&\ii\\ \ii&1\end{pmatrix}\begin{pmatrix}1\\0\end{pmatrix})=\psi(\begin{pmatrix}1\\\ii\end{pmatrix})\in D,$$
which defines the Hodge decomposition 
$$V_{\CC}=\bigoplus V^{p,n-p}$$
and homomorphism $\varphi: S^1= \{z\in\CC\; ;\; |z|=1\}\to G_{\RR}$ of real algebraic groups given by
$$\varphi(z)v=z^{2p-n}v\quad \text{for }v\in V^{p,n-p}.$$
The image $\varphi(S^1)$ is contained in a compact maximal torus $T$ (\cite[Proposition IV.A.2]{GGK}), and we define the Cartan subalgebra $\frakh=\frakt\otimes \CC$.
The associated grading element is $\ttE_{\varphi}= \frac{1}{4\pi\ii}\varphi'(1)\in \frakt$, which satisfies 
$$\ttE_{\varphi}v=\frac{2p-n}{2}v\quad \text{for }v\in V^{p,n-p}.$$

Now $\varphi$ defines a weight-$0$ real Hodge structure on $\frakg$ with the polarization given by minus the Killing form $B$.
Here the $(p,-p)$-component is the eigenspace 
$$\frakg^p=\{X\in\frakg\;|\; \ad{(\ttE_{\varphi})}X=pX\}=\{X\in\frakg\;|\; XV^{k,n-k}\subset V^{k+p,n-k-p}\}.$$
The Lie algebra of the parabolic subgroup stabilizing $\varphi$ is $\frakp=\bigoplus_{p\geq 0}\frakg^{p}.$
Since $\varphi(\ii)$ is the Weil operator, $-B(\varphi(\ii)\bullet , \bullet)$ is positive definite.
Therefore $\Ad{(\varphi(\ii))}$ is an Cartan involution defined over $\RR$, which defines the Cartan decomposition
$\frakg_{\RR}=\frakk_{\RR}\oplus \frakk^{\perp}_{\RR}$
such that
$$\frakk_{\RR}\otimes\CC=\bigoplus_{p\text{: even}}\frakg^{p},\quad \frakk^{\perp}_{\RR}\otimes \CC=\bigoplus_{p\text{: odd}}\frakg^{p}.$$ 
The maximal compact subalgebra $\frakk_{\RR}$ satisfies $\frakt\subset \frakg_{\RR}\cap\frakp\subset \frakk_{\RR}$, i.e. $D$ is measurable in the sense of \cite[\S 4.5]{FHW}.

We set the $sl_2$-triple
$$(\bar{\mathcal{E}},\mathcal{Z},\mathcal{E}):=\Ad_{\bfd_N}(N^+,Y,N)$$
in $\frakg_{\CC}$.
Then 
\begin{align*}
\bar{\calE}=\frac{1}{2}(N+N^+-\ii Y),\quad \calZ =\ii (N-N^+),\quad\calE =\frac{1}{2}(N+N^++\ii Y).
\end{align*}
Since $NF^p\subset F^{p-1}$ by the property of nilpotent orbits, we have
$$\bar{\mathcal{E}}\in\frakg^{1},\quad \mathcal{Z}\in\frakg^0,\quad \mathcal{E}\in\frakg^{-1}.$$

Let $V_{\CC}=\bigoplus V^{\mu}$ be the weight space decomposition with respect to $\frakh$.
That is, $\mu\in\frakh^{*}$ and $v\in V^{\mu}$ if and only if $\xi(v)=\mu(\xi)v$ for all $\xi\in\frakh$.
Then we have 
$
V^{p,n-p}=\bigoplus_{\mu(\ttE_{\varphi})=(2p-n)/2}V^{\mu}.$
The infinitesimal period relation is bracket-generating in the sense of \cite[\S 3.12]{R}.
Then we have
\begin{align}\label{Del-Hdg}
\bfd_N (I^{p,q})=\bigoplus_{\substack{\mu(\ttE_{\varphi})=(2p-n)/2\\ \mu(\mathcal{Z})=p+q-n}}V^{\mu}=\{v\in V^{p,n-p}\;|\; \mathcal{Z}v=(p+q-n)v\}.
\end{align}
by \cite[Theorem 5.5(d)]{R}
 and \cite[Theorem 5.9(b)]{R}.
\subsection{Minimal degenerations}
For a nilpotent orbit $(N,\fil)$, we have the reduced limit 
$$\fil_{\infty}=\lim_{\Im{(z)}\to \infty}\exp{(zN)}\fil\in\bd{(D)}.$$
A nilpotent orbit $(N,\fil)$ is called a minimal degeneration if $\fil_{\infty}$ is in a real codimension-one boundary $G_{\RR}$-orbit of $D$. 
Type of minimal degenerations is classified into two kinds as follows:
\begin{thm}[\cite{GGR} Theorem 1.7]
A minimal degeneration in a period domain is either 
\begin{itemize}
\item[Type I:] $N\neq 0$, $N^2=0$ and $\rank{N}=1,2$, or
\item[Type II:] $N^2\neq 0$, $N^3=0$ and $\rank{N}=2$.
\end{itemize}  
Moreover, type of minimal degeneration is determined by $\{i^{p,q}\}$, where  $i^{p,q}=\dim{I^{p,q}}$, and $\{h^{p,q}\}$ (see also Figure \ref{fig2}):
\begin{itemize}
\item[(I)] $(N,\fil)$ is of type I if and only if there exists $p_o\in\ZZ$ such that $2p_o<n$ satisfying the following conditions:
\begin{itemize}
\item[(i)] $i^{p_o+1,n-p_o},i^{p_o,n-p_o-1}=1$;
\item[(ii)] $i^{p_o,n-p_o}=h^{p_o,n-p_o}-1$ and $i^{p_o+1,n-p_o-1}=h^{p_o+1,n-p_o-1}-1$;
\item[(iii)] for all other $p$ such that $2p<n$, $i^{p,n-p}=h^{p,n-p}$.
\end{itemize}
\item[(II)]  $(N,\fil)$ is of type II if and only if $n=2m$ is even and it satisfies the following conditions:
\begin{itemize}
\item[(i)] $i^{m-1,m-1}, i^{m+1,m+1}=1$;
\item[(ii)] $i^{m-1,m+1}=h^{m-1,m+1}-1$ and $i^{m+1,m-1}=h^{m+1,m-1}-1$;
\item[(iii)] for all other $p$ such that $2p<n$, $i^{p,n-p}=h^{p,n-p}$.
\end{itemize}
\end{itemize}
\end{thm}
\begin{figure}[h]
\begin{center}
\includegraphics{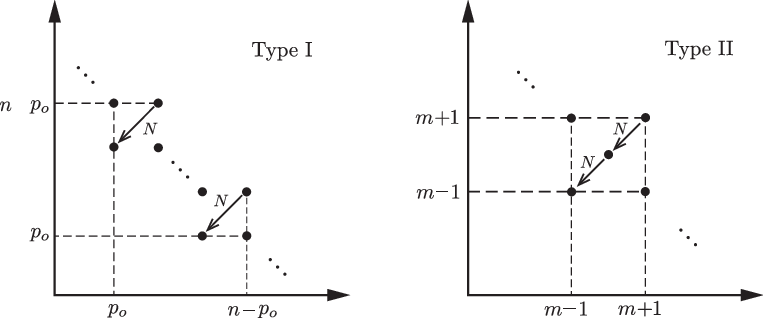}
\caption{Deligne decompositions for minimal degenerations}
\label{fig2}
\end{center}
\end{figure}


Let $(N,\fil)$ be a minimal degeneration such that the limit mixed Hodge structure is $\RR$-split.
In this case, $N$ and $N^+$ are root vectors and $\bfd_N$ is the Cayley transform (see \cite[Appendix to Lecture 10]{G} or \cite{KP}).

\begin{lem}\label{rho}
(1) Suppose that $(N,\fil)$ is of type I. 
Let $v\in I^{p_o+1,n-p_o}$.
We then have
\begin{align*}
&\bfd_N (v)=\frac{1}{\sqrt{2}}(v+\ii Nv), &\bfd_N (Nv) =\frac{\ii}{\sqrt{2}}( v-\ii Nv),\\
&\bfd_N (\bar{v})=\ii \cdot\overline{\bfd_N (Nv)}, &\bfd_N (N\bar{v}) =\ii \cdot\overline{\bfd_N (v)}.
\end{align*}
In particular, if $n=2p_o+1$ and $v\in I^{p_o+1,p_o+1}\cap V_{\RR}$, $\bfd_N(v)=\ii \cdot\overline{\bfd_N (Nv)}$.

(2) Suppose that $(N,\fil)$ is of type II.
Let $v\in I^{m+1,m+1}$.
We then have
\begin{align*}
\bfd_N(v)=\frac{1}{2}v+\frac{1}{2}\ii Nv -\frac{1}{4} N^2v,\quad \bfd_N(Nv)=\ii (v+\frac{1}{2}N^2v),\quad\bfd_N(N^2v)=-2\overline{\bfd_N(v)}
\end{align*}
\end{lem}
\begin{proof}
First, we prove (1).
By property of limit mixed Hodge structure, $N$ is a $(-1,-1)$-endomorphism and the nil-positive element $N^+$  is a $(1,1)$-endomorphism of the limit mixed Hodge structure $(\wf,\fil)$.
Moreover, $N$ is an isomorphism from $I^{p_o+1,n-p_o}$ to the image.
Then
$$(N^++N)v=Nv,\quad (N^++N)^2v=N^+Nv=[N^+,N]v=Yv=v.$$
and we have
\begin{align*}
\bfd_N(v)=&\sum_{\ell =0}^{\infty}\frac{(-1)^\ell}{(2\ell) !}(\frac{\pi}{4})^{2\ell}v+\sum_{k=0}^{\infty}\frac{(-1)^k}{(2k+1) !}(\frac{\pi}{4})^{2k+1}\ii Nv\\
&=\cos{(\frac{\pi}{4})}v+\ii\sin{(\frac{\pi}{4})}Nv=\frac{1}{\sqrt{2}}(v+\ii Nv),\\
\bfd_N(Nv)=&\sum_{\ell =0}^{\infty}\frac{(-1)^\ell}{(2\ell) !}(\frac{\pi}{4})^{2\ell}Nv+\sum_{k=0}^{\infty}\frac{(-1)^k}{(2k+1) !}(\frac{\pi}{4})^{2k+1}\ii v\\
&=\cos{(\frac{\pi}{4})}Nv+\ii\sin{(\frac{\pi}{4})}v=\frac{\ii}{\sqrt{2}}( v-\ii Nv).
\end{align*}
The equations for $\bfd_N (\bar{v})$ and $\bfd_N (N\bar{v})$ follows from similar calculations.

Next, we prove (2).
By property of limit mixed Hodge structure,
\begin{align}\label{cal1}
&N^+Nv=[N^+,N]v=Yv=2v,\nonumber\\
&N^+N^2v=[N^+,N]Nv+NN^+Nv=YNv+2Nv=2Nv.\nonumber\\
&(N^++N)v=Nv,\quad (N^++N)^2v=N^+Nv+N^2v=2v+N^2v,\\
&(N^++N)^3v=2Nv+N^+N^2v=4Nv.\nonumber
\end{align}
Then
\begin{align*}
\bfd_N (v)=v+\sum_{k=0}^{\infty}\frac{(-1)^k}{(2k+1) !}(\frac{\pi}{4})^{2k+1}4^k\ii Nv+\sum_{\ell =1}^{\infty}\frac{(-1)^\ell}{(2\ell) !}(\frac{\pi}{4})^{2\ell}4^{\ell -1}(2v+N^2v).
\end{align*}
Here
\begin{align}\label{cal2}
&\sum_{k=0}^{\infty}\frac{(-1)^k}{(2k+1) !}(\frac{\pi}{4})^{2k+1}4^k=\frac{1}{2}\sum_{k=0}^{\infty}\frac{(-1)^k}{(2k+1) !}(\frac{\pi}{4})^{2k+1}2^{2k+1}=\frac{1}{2}\sin{(\frac{\pi}{2})}=\frac{1}{2}\\
&\sum_{\ell =1}^{\infty}\frac{(-1)^\ell}{(2\ell) !}(\frac{\pi}{4})^{2\ell}4^{\ell -1}=\frac{1}{4}\sum_{\ell =1}^{\infty}\frac{(-1)^\ell}{(2\ell) !}(\frac{\pi}{4})^{2\ell}2^{2\ell }=\frac{1}{4}(\cos{(\frac{\pi}{2})}-1)=-\frac{1}{4}.\nonumber
\end{align}
Therefore,
$$\bfd_N(v)=\frac{1}{2}v+\frac{1}{2}\ii Nv -\frac{1}{4} N^2v.$$
Since we have $(N^++N)Nv=(N^++N)^2v$, by using  (\ref{cal1}) and (\ref{cal2})
\begin{align*}
\bfd_N(Nv) &=\sum_{\ell =0}^{\infty}\frac{(-1)^\ell}{(2\ell) !}(\frac{\pi}{4})^{2\ell}4^{\ell }(Nv)+\sum_{k=0}^{\infty}\frac{(-1)^k}{(2k+1) !}(\frac{\pi}{4})^{2k+1}4^k\ii (2v+N^2v)\\
&=\cos{(\frac{\pi}{2})}Nv+\frac{\ii}{2} (2v+N^2v)=\ii (v+\frac{1}{2}N^2v).
\end{align*}
Since we have $(N^++N)N^2v=N^+N^2v=2Nv$, by using (\ref{cal1}) and (\ref{cal2})
\begin{align*}
\bfd_N(N^2v) &=N^2v+\sum_{k=0}^{\infty}\frac{(-1)^k}{(2k+1) !}(\frac{\pi}{4})^{2k+1}4^k 2\ii Nv+\sum_{\ell =1}^{\infty}\frac{(-1)^\ell}{(2\ell) !}(\frac{\pi}{4})^{2\ell}4^{\ell -1}2(2v+N^2v)\\
&=N^2v+\ii Nv-\frac{1}{2}(2v+N^2v)=\frac{1}{2}N^2v+\ii Nv-v=-2\overline{\bfd_N(v)}.
\end{align*}
\end{proof}

%


\subsection{Pseudoconcave boundary orbits}
For a point in a codimension-one boundary orbit, we may define a local defining function.
Considering its Levi form, we can define pseudoconcavity at a point in a codimension-one boundary orbit as Definition \ref{pcc-dfn}:
\begin{dfn}
A codimension-one boundary $G_{\RR}$-orbit $\mathcal{O}$ in $\bd{(D)}$ is pseudoconcave boundary orbit of $D$ if a holomorphic map $\rho:\mathbb{D}\to \cl{(D)}$ where $\rho(0)\in\mathcal{O}$ and $\bd{(\rho(\mathbb{D}))}\subset D$ exists.   
\end{dfn}
Let $\mathcal{O}$ be a codimension-one boundary $G_{\RR}$-orbit.
We have a minimal degeneration $(N,\fil)$ where the associated reduced limit $\fil_{\infty}$ is in $\calO$.

\begin{thm}\label{prop-per}
Suppose that there exists $r\in \ZZ$ such that $i^{r,n-r}\neq 0$ satisfying one of the following conditions: 
\begin{enumerate}
\item[(i)] $(N,\fil)$ is of type I and $r=p_o+2\ell$ or  $r=p_o-2\ell +1$ ($\ell\geq 1$);
\item[(ii)] $(N,\fil)$ is of type II  and $r=m+2\ell+1$ ($\ell\in\ZZ$).
\end{enumerate}
Then $\calO$ is a pseudoconcave boundary orbit of $D$. 
\end{thm}
\begin{exm}\label{Sp4}
Let $D$ be the period domain with $h^{3,0}=h^{2,1}=1$.
This is the period domain for quintic-mirror threefolds.
In this case, nontrivial nilpotent orbits and the Deligne--Hodge numbers of the limit mixed Hodge structures are classified into the following three types:
\begin{itemize}
\item[(I)] $N^2=0$ and $\dim{(\Im{N})}=1$ ($i^{2,2}=i^{1,1}=i^{3,0}=i^{0,3}=1$);
\item[(II)] $N^2=0$ and $\dim{(\Im{N})}=2$ ($i^{3,1}=i^{1,3}=i^{2,0}=i^{0,2}=1$);
\item[(III)] $N^3\neq 0$ and $N^4=0$ ($i^{3,3}=i^{2,2}=i^{1,1}=i^{0,0}=1$).
\end{itemize} 
Here (I) and (II) are minimal degeneration of type I.
Moreover (I) satisfies the condition of Theorem \ref{prop-per}, and then the corresponding codimension-one boundary $G_{\RR}$-orbit is a pseudoconcave boundary orbit.

\end{exm}
\subsection{Proof of Theorem \ref{prop-per}}
For a nilpotent orbit $(N,\fil)$, there uniquely exists $\tilde{F}^{\bullet}\in \check{D}$ such that the limit mixed Hodge structure $(\wf,\tilde{F}^{\bullet})$ is $\RR$-split by \cite[Proposition 2.20]{CKS}. 
Since $\calO=G_{\RR}F^{\bullet}_{\infty}=G_{\RR}\tilde{F}^{\bullet}_{\infty}$ by \cite[\S 5.1]{KP}, we may choose $\fil\in \check{D}$ so that $\fil=\tilde{F}^{\bullet}$.

We have $\varphi=\bfd_N(\fil)\in D$ as in \S \ref{nilp-sl2}.
Let $K_{\RR}$ be the maximal compact subgroup containing $G_{\RR}\cap P$ where $P$ is the parabolic subgroup of $G$ stabilizing $\varphi$.
We define the base cycle $C_0=K_{\RR}\varphi$.
For $\calE=\Ad_{\bfd_N}(N)$ we have
$$\exp{(\ii\calE)}\varphi=\psi(\frac{1}{\sqrt{2}}\begin{pmatrix}1&\ii\\ \ii&1\end{pmatrix}\begin{pmatrix}1\\\ii\end{pmatrix})=\fil_{\infty}.$$
By Lemma \ref{mainLem}, $\exp{(\ii\calE)}C_0\ni\fil_{\infty}$ and $\exp{(\ii\calE)}C_0\cap C_0\neq \emptyset$.
We may construct $Z\subset gC_0$ and a disk $\mathbb{D}\subset Z$ about $\fil_{\infty}$ satisfying $\bd{(\mathbb{D})}\subset D$ applying the discussion of the proof of \cite[Theorem 3.7]{Huc0}.

Hence we prove Lemma \ref{mainLem}, which completes the proof of Theorem \ref{prop-per}.
\begin{lem}\label{mainLem}
If $r\in \ZZ$ satisfying the condition of Theorem \ref{prop-per} exists, then there exists $k\in K$ such that $\Ad{(k)}(\calE)\in\frakp$.
In particular, $k^{-1}\varphi\in C_0$ is a fixed point of $\exp{(\lambda \calE)}$ with $\lambda\in \CC$.
\end{lem}
\begin{proof}
We first consider the case for type I. 
Let $v\in I^{p_o+1,n-p_o}$.
By Lemma \ref{rho}
\begin{align}\label{decomp}
& \calE \bfd_N(v)=\bfd_N(Nv),\quad \overline{\calE \bfd_N(v)}=-\ii\bfd_N(\bar{v}),\\
&\calE\cdot\overline{\calE \bfd_N(v)}=-\ii\bfd_N(N\bar{v})=\overline{\bfd_N (v)}.\nonumber
\end{align}
In particular, if $n=2p_o+1$ and $v\in I^{p_o+1,p_o+1}\cap V_{\RR}$, then
\begin{align}\label{decomp2}
\bfd_N (v)=\ii\cdot\overline{\calE \bfd_N(v)},\quad \calE\cdot\overline{\calE \bfd_N(v)}=-\ii \calE \bfd_N(v).
\end{align}
We put $w:=\calE \bfd_N(v)$.
We may assume $\|w\|=1$ for the Hodge norm $\|\bullet \| =Q(\varphi(\ii)\bullet , \bar{\bullet})$.
Let $r=p_o+2\ell$ with $\ell\geq 1$ and $i^{r,n-r}\neq 0$.
Since $I^{r,n-r}\subset W^n$ is contained in $\Ker{N}$ by property of monodromy weight filtration, $I^{r,n-r}= \bfd_N(I^{r,n-r})\subset V^{r,n-r}$.
We choose $u\in I^{r,n-r}$ so that $\| u\|=1$.
We define $k\in \Aut{(V_{\CC})}$ by
\begin{align}\label{cpt}
&kw=u,\quad ku=w,\quad k\bar{u}=\bar{w},\quad k\bar{w}=\bar{u},\\
&kv'=v'\quad \text{if } v'\perp w,u,\bar{w}\text{ and }\bar{u} \text{ with respect to }Q.\nonumber
\end{align}
Then $k\in G$, and $\Ad{(\varphi(\ii))}k=k$, i.e. $k\in K$ (see Figure \ref{fig3}).
\begin{figure}[h]
\begin{center}
\includegraphics{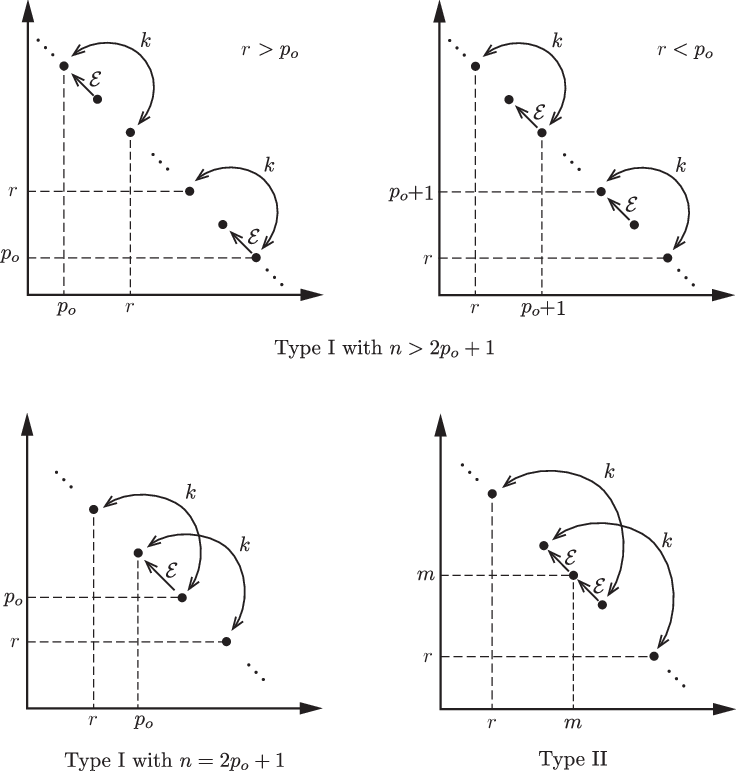}
\caption{Hodge decompositions with respect to $\varphi$}
\label{fig3}
\end{center}
\end{figure}
\\
By using  (\ref{decomp}) and (\ref{decomp2}) we have
\begin{align*}
&\Ad{(k)}(\calE)\bfd_N (v)=
	\begin{cases}
	\ii\Ad{(k)}(\calE)\bar{w}=\ii k\calE \bar{u}=0&\text{if }n=2p_o+1,\\
	kw=u&\text{otherwise};
	\end{cases}\\
&\Ad{(k)}(\calE)\bar{u}=k\calE \bar{w}=
	\begin{cases}
	-\ii kw=-\ii u&\text{if }n=2p_o+1,\\
	\calE \bar{w}=\overline{\bfd_N(v)} &\text{otherwise};
	\end{cases}\\
&\Ad{(k)}(\calE)v'=0\quad\text{if }v'\perp u\text{ and }\overline{\bfd_N(v)}.	
\end{align*}
Since $\bfd_N(v)\in V^{p_o+1,n-p_o-1}$ by (\ref{Del-Hdg}), we conclude that
\begin{align}\label{cal3}
\Ad{(k)}\calE\in 
\begin{cases}
\frakg^{4\ell -1} \subset \frakp&\text{if }n=2p_o+1,\\
\frakg^{2\ell -1}\subset\frakp&\text{otherwise}.
\end{cases}
\end{align}
For the case where $r=p_o-2\ell+1$ with $\ell\geq 1$, we put $w=\bfd_N(v)$ instead and define $k\in K$ as (\ref{cpt}).
Then we obtain (\ref{cal3}).

Next, we consider the case for type II.
Let $v\in I^{m+1,m+1}$ such that $\|\bfd_N(v)\|=1$ and put $w=\bfd_N(v)$.
Then
\begin{align*}
&\overline{\calE w}=-\calE w,\quad -2\bar{w}=\calE^2 w.
\end{align*}
Let $r=m+2\ell+1$ with $i^{r,n-r}\neq 0$.
Since $n-r=m+2(-\ell-1)+1$ and  $i^{n-r,r}\neq 0$, we may assume $\ell<0$.
We choose $u\in V^{r,n-r}$ so that $\| u\|=1$.
We define $k\in K$ as (\ref{cpt}), and then 
\begin{align*}
&\Ad{(k)}(\calE) u=k\calE w=\calE w,\quad \Ad{(k)}(\calE) \calE w=-2\bar{u},\\
&\Ad{(k)}(\calE)v'=0\quad\text{if }v'\perp \bar{u}\text{ and }\overline{\calE w}.	
\end{align*}
Since $\calE w\in V^{m,m}$, we have $\Ad{(k)}\calE\in \frakg^{-2\ell-1}\subset \frakp$.
\end{proof}

%
%

\end{document}